\documentclass{article}
\usepackage{amssymb,amsmath,amsthm}
\usepackage{hyperref}
\usepackage{amsrefs}

\newtheorem{theorem}{Theorem}[section]

\newtheorem{pro}[theorem]{Proposition}
\newtheorem{cor}[theorem]{Corollary}

\theoremstyle{definition}
\newtheorem{definition}[theorem]{Definition}

\theoremstyle{remark}
\newtheorem{remark}[theorem]{Remark}

\numberwithin{equation}{section}

\renewcommand{\to}{\rightarrow}

\newcommand{\restrict}{\upharpoonright}

\newcommand{\da}{\!\downarrow\,}

\begin{document}
\title{Kolmogorov complexity and the Recursion Theorem}

\author{Bj\o rn Kjos-Hanssen\footnote{This material is based upon work supported by the National Science Foundation under Grants No.\ 0652669 and 0901020.}\\
Wolfgang Merkle\\
Frank Stephan}
\maketitle

\begin{abstract}
	Several classes of diagonally non-recursive (DNR) functions are characterized in terms
	of Kolmogorov complexity. In particular, a set of natural numbers
	$A$ can wtt-compute a DNR function iff there is a nontrivial
	recursive lower bound on the Kolmogorov complexity of the
	initial segments of $A$. Furthermore, $A$ can Turing
	compute a DNR function iff there is a nontrivial $A$-recursive
	lower bound on the Kolmogorov complexity of the initial segements
	of $A$. $A$ is PA-complete, that is, $A$ can compute a
	$\{0,1\}$-valued DNR function, iff $A$ can compute a function $F$
	such that $F(n)$ is a string of length $n$ and maximal $C$-complexity
	among the strings of length $n$. $A \geq_T K$ iff $A$ can compute a
	function $F$ such that $F(n)$ is a string of length $n$ and maximal
	$H$-complexity among the strings of length $n$. Further characterizations
	for these classes are given.
	The existence of a DNR function in a Turing degree is equivalent to the
	failure of the Recursion Theorem for this degree; thus the provided
	results characterize those Turing degrees in terms of Kolmogorov
	complexity which do no longer permit the usage of the Recursion Theorem.
\end{abstract}

\section{Introduction}
\noindent
The Recursion Theorem can be stated in two ways: First one can
say that every recursive function $f$ has a fixed-point with
respect to any given acceptable numbering of all r.e.\ sets:
$\exists e\,[W_e = W_{f(e)}]$. Second one can say that
every total recursive function $f$ coincides
at some places with the diagonal function:
$\exists e \,[\varphi_e(e) \da = f(e)]$. Jockusch \cite{Jockusch:89}
showed that these two variants of the Recursion Theorem are also
equivalent relative to any oracle $A$:
Every function $f \leq_T A$ admits a fixed-point iff every function
$g \leq_T A$ coincides with the diagonal function somewhere.
Special attention has been given to the oracles $A$ which
permit to avoid the Recursion Theorem and the topic of the
the present work is to relate these oracles to notions of
Kolmogorov complexity; namely to the below introduced classes
of complex and autocomplex sets where one can give nontrivial
lower bound on the complexity of initial segments of the set
observed. The formal definitions of fixed-point
free and diagonally nonrecursive functions are as follows.

\begin{definition}
A function $f$ is called fixed-point free if
$W_x \neq W_{f(x)}$
for all $x$, where $W_x$ is the $x$th recursively enumerable set.
The partial recursive function mapping $x$ to $\varphi_x(x)$ whenever
defined is called the diagonal function, where $\varphi_x$ is the
$x$th partial recursive function. A function $g$ is called
diagonally nonrecursive $($DNR$)$
iff it is total and differs from the diagonal function on its domain.
\end{definition}

\noindent
We also relate complexity to sparseness: Suppose sets $A$ and $B$
are given such that $A \leq_{wtt} B$ and $B$ is hyperimmune.
Then $A$ is computed by querying rather few bits from $B$ and $B$ has
very long intervals consisting of zeroes only. So when computing $A(x)$
for many inputs $x$, the oracle $B$ will answer ``$0$'' on all questions
that have not been asked before, hence $B$ will not be of much
use. We show that this can happen for some $B$ just in case $A$ has low
Kolmogorov complexity in a certain sense. This is again equivalent
to the statement that the Recursion Theorem applies to recursively
bounded functions wtt-reducible to $A$.

A further topic investigated is to determine how difficult it is to
compute for each $n$ a string of maximal Kolmogorov complexity within
$\{0,1\}^n$. While for the case of plain Kolmogorov complexity the answer
to this problem depends on the universal machine under consideration,
in the case of prefix-free Kolmogorov complexity the answer is that
the problem to compute such strings is as difficult as to solve the
halting problem.

Besides this we study related questions and are able to characterise
the oracles which are high or autocomplex; furthermore, related
characterisations are obtained for the oracles which are high or
PA-complete. In the last section we study the notion of r.e.\ traceable
sets which are, roughly spoken, some type of notion opposite to the notion of autocomplex sets.

\section{Complex sets}

\noindent
A set will be called {\em complex} if the prefixes of its characteristic
function have a nontrivial lower bound computed by some recursive function;
this notion is a weakening of the corresponding characterization
for {\em randomness} where this lower bound is just the length of the
prefix minus some constant. Recall that $C$ is the plain and $H$ the
prefix-free Kolmogorov complexity.

Recall that a function $g:\omega\to\omega$ is called an \emph{order function} if it is recursive, nondecreasing, and unbounded. For a set $A$, we say that the function $g:\omega\to\omega$ is an $A$-\emph{order function} if it is recursive in $A$, nondecreasing, and unbounded.

\begin{definition}
A set $A$ is \emph{complex} if there is an order function $g$ such that
$C(A\restrict y)\ge g(y)$ for all $y$.

A set $A$ is \emph{autocomplex}
if there is an $A$-order function $g$ such that $C(A\restrict y)\ge g(y)$ for all $y$.
\end{definition}

Here $A\restrict x$ is the finite
binary sequence $A(0)\ldots A(x-1)$.

\begin{pro}\label{prop:autocomplex}
For any set~$A$, the following conditions are equivalent.
\begin{itemize}
\item[{\rm(1)}]
The set~$A$ is autocomplex.
\item[{\rm(2)}]
There is an $A$-recursive function~$h$ such that for all~$n$, 
${C}(A\restrict h(n))\ge n$.
\item[{\rm(3)}]
There is an $A$-recursive function~$f$ such that for all~$n$, ${C}(f(n))\ge n$. 
\end{itemize}
\end{pro}
\begin{proof}
We show~$(1) \Rightarrow (2) \Rightarrow (3) \Rightarrow (1)$.
Given an autocomplex set~$A$, choose an $A$-recursive order~$g$ where $C(A\restrict n)\ge g(n)$ and in order to obtain a function~$h$ 
as required by~(2), let
\[
h(n) = \min\{l  \colon   g(l) \ge n  \} .
\]
Given a function~$h$ as in~(2), in order to obtain a function~$f$ as required by~(3), simply let~$f(n)$ be equal to (an appropriate encoding of) the prefix of~$A$ of length~$h(n)$. Finally, given an $A$-recursive function~$f$ as in~(3), let~$u(n)$ be an $A$-recursive order such that some fixed oracle Turing machine~$M$ computes~$f$ with oracle~$A$ such that~$M$ queries on input~$n$ only bits~$A(m)$ of A where~$m\le u(n)$. Then for any~$l \ge u(n)$, 
the value of~$f(n)$ can be computed from~$n$ and~$A\restrict {l}$, hence
\[
n \le {C}(f(n)) \le^{+} {C}(A\restrict {l}) + 2 \log n ,
\]
and thus for almost all~$n$ and all~$l \ge u(n)$, we have~$n/2 \le {C}(A\restrict {l})$. As a consequence, a finite variation of the $A$-recursive order
\[g\colon n \mapsto \max\{l \colon u(l) \le n \}/2\]
witnesses that~$A$ is autocomplex.
\end{proof}

\noindent
The next result shows that the notions of complex and autocomplex sets
can be characterized in terms of the degrees of DNR functions. In the
case of complex sets, one can recast this notion in terms of the
degrees of recursively bounded DNR functions. because a complex set can tt-compute
only functions which are recursively bounded. There are Turing degrees
which contain DNR functions but no recursively bounded functions \cite{AKLS},
hence some Turing degrees consist of autocomplex sets
without containing any complex set.

\begin{theorem} \label{th:complex}
Let $A$ be any set. 
\begin{itemize}
\item[\rm 1.] $A$ is autocomplex iff $A$ computes a DNR function.
\item[\rm 2.] $A$ is complex iff $A$ $tt$-computes a DNR function
          iff $A$ $wtt$-computes a DNR function.
\end{itemize}
\end{theorem}

\begin{proof}
1. Suppose $A$ is not autocomplex, but computes a DNR function
$\varphi_{r}^A$. There is a recursive function $e(\cdot)$ such that
for any $x$ and $y$, if the universal machine $U$ on $x$ converges,
and additionally $\varphi^{U(x)}_{r}(y)$ converges, then
$\varphi_{e(x)}(y)=\varphi^{U(x)}_{r}(y)$.
For each $n$, let $g(n)$ be the maximum of the use of all computations
$\varphi_{r}^A(e(x))$ with
$|x|\le n$. Clearly $g$ is recursive in $A$. Since $A$ is not
autocomplex, there must be some $n$ such that $C(A\restrict g(n))\le
n$. If $x_n$ is a witnessing code, i.e. $U(x_n)=A\restrict g(n)$,
$|x_n|\le n$, then $\varphi_{r}^A(e(x_n))=\varphi_{e(x_n)}(e(x_n))$.
Thus $\varphi_r^A$ is not
DNR.

Now suppose $A$ is autocomplex. Let $g$ be an $A$-order function such that for all $n$, $C(A\restrict n)\ge g(n)$. 
Let $\gamma$ be an $A$-recursive function such that
$g(\gamma(n))\ge n$ for all $n$. Let $f(n)=A\restrict \gamma(n)$; $f$
is recursive in $A$. If $f(n)$ is infinitely often equal to
$\varphi_{n}(n)$ then for such $n$, $C(A\restrict
\gamma(n))\le \log n+c$ for some constant $c$. Letting $m=\gamma(n)$, 
\[
C(A\restrict m)\le\log n+c \le \log g(m)+c<g(m)
\]
for sufficiently
large $m$, which is a contradiction. Thus, we conclude that
$f(n)\ne\varphi_n(n)$ for all but finitely many $n$, so $f$ computes a
DNR function.

2. This is obtained by the same proof; if $A$ $wtt$-computes
$\varphi_r^A$ then the function $g$ is recursive, and if $g$ is
recursive then $f$ is $tt$-computable from $A$. \end{proof}

\noindent
Theorem~\ref{th:complexchar} below gives a variety of useful
characterizations of the complex sets. These characterizations
are mainly based on the following definition and result of Jockusch.

\begin{definition}[Jockusch \cite{Jockusch:89}]
A function $h:\omega\to\omega$ is called \emph{strongly DNR (SDNR)}
if for all $x$, $h(x)\ne\varphi_y(y)$ for all $y\le x$.
\end{definition}

\begin{theorem}[Jockusch \cite{Jockusch:89}]\label{nov192006}
Every DNR function computes an SDNR function.
\end{theorem}

\noindent
The characterizations of complex sets and Jockusch's result can
be extended to the following comprehensive characterization of
the complex sets.

\begin{theorem} \label{th:complexchar}
For a set $A$ the following conditions are equivalent.
\begin{enumerate}
\item[\rm(1)]
There is a SDNR function $f_1 \leq_{wtt} A$;
\item[\rm(2)]
For every recursive function $g$ there is a function
$f_2 \leq_{wtt} A$ such that for all recursive $h$
with $\forall n\,[|W_{h(n)}| < g(n)]$ it holds that
$\forall^{\infty} n\,[f_2(n) \notin W_{h(n)}]$;
\item[\rm(3)]
There is $f_3 \leq_{wtt} A$ such that for all partial-recursive
$V$ with domain $\{0,1\}^*$ and almost all $n$,
$f_3(n) \notin \{V(p): p \in \{0,1\}^* \wedge |p|<n\}$;
\item[\rm(4)]
There is $f_4 \leq_{wtt} A$ with $C(f_4(n)) \geq n$ for all $n$;
\item[\rm(5)]
There is a DNR function $f_5 \leq_{wtt} A$;
\item[\rm(6)]
$A$ is complex.
\end{enumerate}
One can also write $f_k \leq_{tt} A$ in place
of $f_k \leq_{wtt} A$ in these conditions.
\end{theorem}

\begin{proof}
(1) implies (2): Let a recursive function $g$ be given,
without loss of generality $g$ is increasing. Furthermore
there is a recursive function $a$ such that
function $a$ such that $\varphi_{a(e,n,m)}(x)$ is the $m$-th element
enumerated into $W_{\varphi_e(n)}$ whenever $|W_{\varphi_e(n)}| \geq m$.
There is a recursive function $b$ such that $b(k) > a(e,n,m)$ for
all $e,m,n \leq g(k)$. Now let $f_2(n) = f_1(b(n))$.
It is follows from the definition of SDNR functions that
for all total function $h = \varphi_e$ and all $n>e$,
$f_1(b(n))$ differs from the first $g(n)$ elements of $W_{h(n)}$;
in the case that $|W_{h(n)}| < g(n)$ it follows that
$f_2(n) \notin W_{h(n)}$.
Furthermore, as $f_1 \leq_{wtt} A$ and
$f_2$ is many-one reducible to $f_1$, $f_2 \leq_{wtt} A$ as well.
Hence $f_2$ satisfies the requirements asked for.

(2) implies (3): Here one can choose $g$ such that
$g(n) = 2^n$ and consider the corresponding $f_2$. Then one takes
$f_3 = f_2$ and 
\[
W_{h(n)} = \{V(p): p \in \{0,1\}^* \wedge |p| < n\}.
\]
By assumption $f_2(n) \notin W_{h(n)}$ for almost all $n$,
the same holds then also for $f_3$.

(3) implies (4): Let $f_3$ be given as in (3) and let $U$ be
the universal machine on which $C$ is based. Then there is a constant
$c$ such that for all $n \geq  c$ and all $p$ in the domain of
$U$ with $|p| < n$, $f_3(n) \neq U(p)$. Thus $C(f_3(n)) \geq n$ for
all $n \geq c$. Now let $f_4(n) = f_3(n+c)$ in order to meet condition (4).

(4) implies (5): There is a constant $c$ with $C(\varphi_n(n)) < n+c$
for all $n$ such that $\varphi_n(n)$ converges. Taking $f_4$ as in (4), the function $f_5$ given
as $f_5(n) = f_4(n+c)$ satisfies that $C(\varphi_n(n)) < n+c \leq C(f_5(n))$
for all $n$ where $\varphi_n(n)$ is defined. Hence (5) is satisfied.

(5) implies (6): This follows from Theorem~\ref{th:complex}.

(6) implies (1): This is similar to the proof in Theorem~\ref{th:complex}.
Just choose again a constant $c$ such that $C(\varphi_n(n)) \leq n+c$
whenever defined. It follows from the definition of complex that
there is a recursive function $f_6$ with $C(A \restrict f_6(n)) > n+c$
for all $n$ and hence $f_1(n) = (A \restrict f_6(n))$ is an
SDNR function.

This completes the equivalences. As all steps from $(k)$ to $(k+1)$
for $k=1,2,3,4$ make a many-one reduction from one function to
another and as $f_1 \leq_{tt} A$, these functions are all tt-reducible
to $A$ along the lines of the proof. Hence all conditions hold with
$f_k \leq_{tt} A$ in place of $f_k \leq_{wtt} A$.
\end{proof}

\noindent
One can easily see that above characterizations hold with Turing
reducibility in place of weak truth-table reducibility for autocomplex
sets. Hence one obtains the following theorem as well.

\begin{theorem} \label{th:autocomplex}
For a set $A$ the following conditions are equivalent.
\begin{enumerate}
\item[\rm(1)]
There is a SDNR function $f_1 \leq_{T} A$;
\item[\rm(2)]
For every recursive function $g$ there is a function
$f_2 \leq_T A$ such that for all recursive $h$
with $\forall n\,[|W_{h(n)}| < g(n)]$ it holds that
$\forall^{\infty} n\,[f_2(n) \notin W_{h(n)}]$;
\item[\rm(3)]
There is $f_3 \leq_T A$ such that for all partial-recursive
$V$ with domain $\{0,1\}^*$ and almost all $n$,
$f_3(n) \notin \{V(p): p \in \{0,1\}^* \wedge |p|<n\}$;
\item[\rm(4)]
There is $f_4 \leq_{T} A$ with $C(f_4(n)) \geq n$ for all $n$;
\item[\rm(5)]
There is a DNR function $f_5 \leq_{T} A$;
\item[\rm(6)]
$A$ is autocomplex.
\end{enumerate}
\end{theorem}

\begin{remark} \rm
In the two preceeding theorems, condition (2) said that one can
compute for each order a function avoiding all r.e.~traces of
given cardinality almost everywhere; the DNR Turing degrees
can also be characterized as those where there is for each
given r.e.\ trace a function avoiding this trace everywhere.
Condition (3) is the r.e.\ counterpart of Theorem~\ref{SNRisDNRorhigh}~(5)
below. Condition (4) says that one can compute a function which takes
on all inputs a value of sufficiently high Kolmogorov complexity.
Condition (6) is already known to be equivalent to the other ones,
but it was convenient to go the proof from (5) to (1) through (6)
in Theorem~\ref{th:complexchar}.
\end{remark}

\begin{remark} \rm
Instead of computing a lower bound $h$ for the complexity of $A$
from $A$ as an oracle, one can also formulate the same result
by making a function $h$ mapping strings to lower bounds; this
function then needs only to be correct on strings stemming from
the characteristic function of $A$. The characterization is the
following:

$A$ is complex iff there is a recursive function $h$ such that
\begin{itemize}
\item for all $\sigma,\tau \in \{0,1\}^*$, $h(\sigma\tau) \geq h(\sigma)$ and
\item for all $y$, $h(A\restrict y) \leq C(A\restrict y)$.
\item for every $n$ the set of strings
$\sigma \in \{0,1\}^*$ with $h(\sigma) \leq n$ is finite.
\end{itemize}
Similarly $A$ is autocomplex iff there is a function $h\le_T A$ with these same properties.
\end{remark}

\begin{remark}\label{Kanovich}\rm
M.I.\ Kanovich \cites{Kanovich:70, Kanovich:69} (see Li and 
Vitanyi \cite{LiVitanyi3}, Exercise 2.7.12, p.\ 184) states a result to the effect that the notions of being complex and autocomplex, defined in terms of monotonic complexity, are the same as being wtt-complete and T-complete in the special case of recursively enumerable sets.  
\end{remark}

\section{Hyperavoidable and effectively immune sets}

Miller \cite{Miller:02} introduced the notion of \emph{hyperavoidable
set}. A set $A$ is hyperavoidable iff it differs from all characteristic
functions of recursive sets within a length computable from a program
of that recursive set. For random sets, this length is at most the
length of the program plus a constant. So hyperavoidable sets are
a generalization of random sets and Theorem~\ref{bjornsidt} shows that
one can characterize that $A$ is hyperavoidable similarly
to the way one characterizes that $A$ is random in terms of
prefix-free Kolmogorov complexity: $A$ is hyperavoidable iff
there is an order function $g$ such that $\forall
x\,(C(A\restrict x) \geq g(x))$. 

\begin{definition}
A set of nonnegative integers $A$ is called hyperavoidable if there is
an order function $h$ such that for all $x$
with
\[
	\{0,1,\ldots,h(x)-1\}\subseteq\{y:\varphi_x(y)\da\in\{0,1\}\},
\]
we have $A\restrict h(x)\neq \varphi_x \restrict h(x)$. In other words,
\[
\forall x\,\exists y<h(x)\,\,A(y)\ne\varphi_x(y).
\]
\end{definition}

\noindent
Note that if $A$ is hyperavoidable via $h$ and $\tilde h$ is a further
order function with
$\tilde h(x) \geq h(x)$ for all $x$ then $A$ is also hyperavoidable
via $\tilde h$.

The original reason for interest in hyperavoidability is the following result.

\begin{theorem}{\rm\cite{Miller:02}*{Theorem 4.6.4}}
A set is hyperavoidable iff it is not wtt-reducible to any
hyperimmune set.
\end{theorem}

\begin{theorem}\label{bjornsidt}
For a set $A$ the following statements are equivalent:
\begin{itemize}
\item[\rm(1)] $A$ is complex.
\item[\rm(2)] $A$ is hyperavoidable.
\end{itemize}
\end{theorem}
\begin{proof} 
(1) implies (2):
If $A$ is complex then as we have seen, $\varphi_r^A$ is a DNR
function for a wtt-reduction
$\varphi_r$. Assume that that $u(x)$ is the use of the wtt-reduction
at input $x$,
that is, the maximal element queried at the computation of $\varphi^A_r(x)$;
note that this element is independent of the oracle $A$.
Furthermore, for any $n$ let $W_n = \{z : \varphi_n(z) \da = 1\}$.
Now let $\varphi_{s(n)}(x) = \varphi_r^{W_n}(x)$ iff
$\varphi_n(z) \da \in \{0,1\}$ for all $z \leq u(x)$ and
the computation $\varphi_r^{W_n}(x)$ terminates;
let $\varphi_{s(n)}(x)$ be undefined otherwise. Define
$h(n) = u(s(n))+1$. Since $\varphi_r^A$ is a DNR function there is no
$n$ such that $\varphi_{s(n)}(s(n)) \da = \varphi_r^A(s(n))$. Thus
for every $n$ there is a $y \leq u(s(n))$ such that $\varphi_n(y) \uparrow$
or $\varphi_n(y) \neq A(y)$. It follows that for every $n$ the
function $\varphi_n$ differs from the characteristic function
of $A$ before $h(n)$.

(2) implies (1): Suppose $A$ is not complex. Let $U$ be the universal function
on which $C$ is based.
Let $f$ be a total recursive function such that, for all $\sigma \in \{0,1\}^*$
where $U(\sigma)$ is defined, $\varphi_{f(\sigma)}$ is the
characteristic function of the set $\{x: U(\sigma)(x) \da = 1\}$.
Furthermore, let $\tilde f(n) = \max\{f(\sigma) : \sigma \in \{0,1\}^*
 \wedge |\sigma| \leq n\}$.

Let $h$ be any order function. Let
$\tilde h(n)$ be the maximal $m$ with $m=0 \vee h(\tilde f(m)) \leq n$.
Then $\tilde h$ is also an order function.
Thus there is a $y$ such that $C(A \restrict y) < \tilde h(y)$.
Then there is a program $\sigma$ for the universal machine $U$
with $U(\sigma) = A \restrict y$ and $|\sigma| < \tilde h(y)$.
It follows that $h(f(\sigma)) < y$. Thus the characteristic functions of
$A$ and $\{x: U(\sigma)(x) \da = 1\}$ both coincide with
$\varphi_{f(\sigma)}$ on the first $y$ inputs and so $A$
is not hyperavoidable via $h$. Since the choice of $h$ was
arbitrary, $A$ is not hyperavoidable.
\end{proof}

\noindent
Miller \cite{Miller:02} investigated the relation between
hyperavoidable sets and effectively immune sets. A set $A$ is \emph{immune}
if it has no infinite recursive subset and
\emph{effectively immune} if there is a partial recursive function $\psi$
such that for all $e$, if $W_e \subseteq A$ then $e$ is in the domain
of $\psi$ and $|W_e| \leq \psi(e)$.
If there is no recursive function $f$ with
$\forall n \,(|A \cap \{0,1,2,\ldots,f(n)\}| \geq n)$
then $A$ is called \emph{hyperimmune}.

\begin{theorem}{\rm\cite{Miller:02}*{Theorem 4.5.3}}\label{eff}
Any effectively immune, nonhyperimmune set is hyperavoidable.
\end{theorem}

\noindent
The converse is false; if $A$ is a complex set then $A\oplus\omega$
is still complex, hence hyperavoidable, but not immune, hence certainly
not effectively immune. 
However, up to truth-table degree, we shall see that the converse does hold.

\begin{theorem}
If $A$ is hyperavoidable then there is a set $B \equiv_{tt} A$ which
is effectively immune but not hyperimmune. This set can be viewed as a set
of strings and is given as $B = \{A(0)A(1)\ldots A(n): n \in \omega\}$.
\end{theorem}

\begin{proof}
Obviously $B \equiv_{tt} A$. Furthermore,
$B$ is not hyperimmune as $B$ contains a binary string of every
length. Furthermore, it is effectively immune: As $A$ is complex there
is a recursive function $\psi$ such that the complexity of all initial
segments of $A$ which are longer than $\psi(n)$ is above $2n$. Hence,
whenever $|W_e| > \psi(e)$ and $W_e \subseteq B$, one can find
effectively in $e$ a string $\theta(e) \in W_e$ which has at least
the length $\psi(e)$. Now $C(\theta(e)) \leq e+c$ for some constant
$c$ and all $e$ in the domain of $\theta$. On the other hand
$C(\theta(e)) \geq 2e$ for all $e$ in the domain of $\theta$
with $W_e \subseteq B$; so it can only happen for $e \leq c$
that $W_e \subseteq B \wedge |W_e| \geq \psi(e)$. A finite
modification of $\psi$ makes $\psi$ to be a witness for
$B$ being effectively immune.
\end{proof}

\noindent
Note that the proof actually shows that $B$ is strongly effectively
immune. Hence one gets the following corollary.

\begin{cor}
A set is complex iff its truth-table degree contains a set which is
strongly effectively immune but not hyperimmune.
\end{cor}

\begin{remark} \rm 
It is well-known that a Turing degree contains a DNR function iff
it contains an effectively immune set. In other words, a set is
autocomplex iff its Turing degree contains an effectively immune set.
\end{remark}

\section{Completions of Peano arithmetic}

\noindent
Theorem~\ref{th:complexchar}~(4) shows that DNR is equivalent to
the ability to compute a function $F$ such that
$C(F(n)) \geq n$ for all $n$. The next result shows that if
one enforces the additional constraint $F(n) \in \{0,1\}^n$ then
one obtains the smaller class of PA-complete degrees
instead of the DNR ones. Recall that $A$ has PA-complete
degree iff $A$ computes a DNR function with a finite range.
As Jockusch~\cite{Jockusch:89} showed, one can specify this
range to be any given finite set as long as this set has
at least $2$ elements.

\begin{theorem} \label{pacharacterization}
The following are equivalent for every set $A$.
\begin{enumerate}
\item[\rm (1)]
$A$ computes a lower bound $B$ of the plain complexity $C$ such
that for all $n$ there are at most $2^n-1$ many $x$ with $B(x)<n$.
\item[\rm (2)]
$A$ computes a function $F$ such that for all $n$, $F(n)$ has length $n$
and satisfies $C(F(n)) \geq n$.
\item[\rm (3)]
$A$ computes a DNR function $D$ which has a fixed finite set as range.
\end{enumerate}
\end{theorem}

\begin{proof}
(1) implies (2): $F(n)$ is just the lexicographically first string
$y$ of length $n$ such that $B(y) \geq n$. This string exists by
the condition that there are at most $2^n-1$ strings $x$ with
$B(x) < n$. Since $B$ is a lower bound for $C$, one has that
$C(F(n)) \geq n$ for all $n$. Furthermore, $F$ is computed from $B$.

(2) implies (3): There is a partial recursive function $\psi$ such that
$\psi(x) = x \varphi_n(n)$ if $n$ is the length of $x$ and $\varphi_n(n)$
is defined. Furthermore there is a constant $c$ such that $C(\psi(x)) < n+c$
for all $x,n$ with $x \in \{0,1\}^n$. Now one defines that $D(n)$ consists
of the last $c$ bits of $F(n+c)$; this function is computed from $F$.
Let $x$ be the first $n$ bits of $F(n+c)$ and assume that $\varphi_n(n)$
is defined. Then $C(\psi(x)) < n+c$ and
$x D(n) = F(n+c) \neq \psi(x) = x \varphi_n(n)$. Thus $D$ is a
DNR function and its range is the finite set $\{0,1\}^c$.

(3) implies (1): Since $D \leq_T A$ is a DNR function with a
finite range, the set $A$ is PA-complete by a result of
Jockusch \cite{Jockusch:89}.
Thus there is a set $G \leq_T A$ which extends the graph of the
universal function $U$ on which $C$ is based. $G$ satisfies the
following two $\Pi^0_1$ conditions:
\begin{itemize}
\item
$\forall p,x,s\,(U(p) \da = x \text{ at stage }s \Rightarrow (p,x) \in G)$;
\item
$\forall p,x,y\,((p,x) \in G \wedge (p,y) \in G \Rightarrow x = y)$.
\end{itemize}
Now one defines $B(x) = \min\{|p|: (p,x) \in G\}$. By standard
Kolmogorov complexity arguments, it follows that $B$ is a lower
bound for $C$ and that there are at most $2^n-1$ many $x$ with
$B(x) < n$ for all $n$.
\end{proof}

\noindent
One might ask whether one can strengthen condition (2) in
Theorem~\ref{pacharacterization} and actually compute a string of maximal
plain Kolmogorov complexity for every length $n$ from any PA-complete
oracle. The answer to this question is that it depends on the universal
machine on which the plain complexity is based. That is, for every
r.e.\ oracle $B$ one can compute a corresponding universal machine
which makes this problem hard not only for PA but also for~$B$.

\begin{theorem} \label{plainmaximal}
For every recursively enumerable oracle $B$ there is a universal
machine $U_B$ such that the following two conditions are equivalent
for every oracle $A$:
\begin{itemize}
\item $A$ has PA-complete degree and $A \geq_T B$.
\item There is a function $F \leq_T A$ such that for all $n$
      and for all $x \in \{0,1\}^n$,
      $F(n) \in \{0,1\}^n$ and $C_B(F(n)) \geq C_B(x)$, where $C_B$ is
      the plain Kolmogorov complexity based on the universal machine $U_B$.
\end{itemize}
\end{theorem}

\begin{proof}
Given any universal machine $U$ and r.e.\ set $B$, the value
$U_B(p)$ takes the first case where there is a $q$ satisfying
the corresponding condition:
\begin{itemize}
\item If $p=0q$ and $U(q)$ is defined and has output $x$ and $|x|>|p|+2$
             then $U_B(p) = x$.
\item If $p = 10q$ and $q \in \{0,1\}^*\cdot \{0\}$ and $|q| \in B$ then
             $U_B(p) = q$.
\item If $p = 110q$ and $q \in \{0,1\}^*\cdot \{1\}$ then $U_B(p) = q$.
\item If $p = 1110q$ then $U_B(p) = q$.
\item In all other cases, $U_B(p)$ is undefined.
\end{itemize}
First it is verified that $U_B$ is a universal machine.
If $C(x) \geq |x|-3$ then $C_B(x) \leq C(x)+7$ since
$U_B(1110x)=x$ and $C_B(x) \leq |x|+4$ for all $x$.
If $C(x) < |x|-3$ then there is a program
$q$ with $U(q)=x \wedge |q| < |x|-3$. Taking $p = 0q$,
one has $|x|>|p|+2$ and $U_B(p) = x$. So
$C_B(x) \leq C(x)+1 \leq C(x)+7$ again. Thus
$U_B$ is a universal machine and $C_B$ a legitimate choice
for the plain Kolmogorov complexity.

Note that for $U_B$, there are at most $2^{n-1}-1$ many
strings $x$ of length $n$ with $C_B(x)<n-1$. Assume that $n>0$,
$x \in \{0,1\}^n$ and $C_B(x) \geq n-1$. There are three cases:
\begin{itemize}
\item If $x \in \{0,1\}^{n-1}\cdot \{1\}$ then $C_B(x) = n+3$.
\item If $x \in \{0,1\}^{n-1}\cdot \{0\}$ and $n \notin  B$
      then $C_B(x) = n+4$.
\item If $x \in \{0,1\}^{n-1}\cdot \{0\}$ and $n \in  B$
      then $C_B(x) = n+2$.
\end{itemize}
Note that for each length $n$ there are strings ending with
$0$ and ending with $1$ which satisfy $C_B(x) \geq |x|-1$.
So given $n$, let $x \in \{0,1\}^n$ have maximal plain
complexity $C_B$.
\begin{itemize}
\item If $n \in B$ then
      $x \in \{0,1\}^{n-1}\cdot \{1\}$ and $C_B(x) = n+3$.
\item If $n \notin B$ then
      $x \in \{0,1\}^{n-1}\cdot \{0\}$ and $C_B(x) = n+4$.
\end{itemize}
On one hand, if $F \leq_T A$ and $F$ is as defined in the
statement of this theorem and $n>0$ then the last bit of $F(n)$
is equal to $B(n)$. Thus $B \leq_T A$. Furthermore,
$A$ is PA-complete since $C_B(F(n)) \geq n$ for all $n$.

On the other hand,
if $B \leq_T A$ then one can compute the maximal plain complexity
of a string of length $n$ which is $n+4-B(n)$. Having this number,
one can use the PA-completeness of $A$ to find a string $x$ of length
$n$ such that $C_B(x) \geq n+4-B(n)$. This gives the desired
equivalence.
\end{proof}

\noindent
While PA-completeness can be characterized in terms
of $C$, the obvious analogues fail for $H$. First
one cannot replace $C$-incompress\-ible by $H$-incompress\-ible
since one can compute relative to any random set $A$
the function mapping $n$ to the $H$-incompressible string
$A(0)...A(n)$. So one would like to know whether oracles
$A$ which permit to compute strings of maximal complexity
would characterize the PA-complete degrees. But, instead,
the corresponding notion gives a characterization of the
halting problem~$K$. Note that the following theorem is
independent of the underlying universal machine.

\begin{theorem} \label{th:calude}
For any set $A$, $A \geq_T K$ iff there is a function
$F \leq_T A$ such that $\forall n \,\forall x \in \{0,1\}^n\,
(F(n) \in \{0,1\}^n \wedge H(x) \leq H(F(n)))$.
\end{theorem}

\begin{proof}
Since $H$ is $K$-recursive, such an $F$ can obviously be computed if
$A \geq_T K$. For the remaining direction, assume that $F$ is
as in the statement of the theorem and $F \leq_T A$. The proof
consists now of three parts:
\begin{itemize}
\item First, a sequence of partitions is constructed to be used later.
\item Second, it is shown that there is a constant $k$ such that for
      every $n$ and every $m > n+k+1$ the binary
      number $bv(y)$ consisting of the last $2k$ bits $y$
      of the string $F(2^m+n)$ satisfies that $P_{m,bv(y)}$
      does not contain $H(n)$.
\item Third, it is shown how the fact from the second statement
      can be used to prove that $K \leq_T A$.
\end{itemize}
First, let $P_0,P_1,P_2,\ldots$ be an enumeration of all
primitive recursive permutations of the integers and let
$P_{m,o}$ be the $o$-th member of the permutation; here
$P_{m,o} = \emptyset$ in the case that the permutation has
less than $o$ nonempty members.
Note that every partition has infinitely many indices.

Second, let $U$ be the universal machine on which $H$ is based
and let $\tilde U(p) = n$ whenever $U(p) = 2^m+n$ for some $m>n$.
It is known \cite{LiVitanyi2}*{Section 4.3} that there is a constant $c_1$ such that
\[
   2^{-H(n)} \geq \sum_{p \mbox{ \rm \scriptsize with }
      \tilde U(p) = n} 2^{-|p|-c_1}
\]
for all $n$. Now, let $n$ be given and
$p$ be of length $H(n)$ such that $U(p)=n$,
that is, let $p$ be a minimal program for $n$. As the sum
of $2^{-|q|-c_1}$ over $q$ with $\exists m>n\,(U(q) = 2^m+n)$
is bounded by $1$, there is, uniformly in $p$, a prefix-free
machine $V_p$ which is conditionally universal in the following
sense: for all $m>n$ there is a $q$ with
$|q| = H(2^m+n)+c_1-|p| \wedge V_p(q)=2^m+n$.
If $p$ is not a minimal program for
any $n$ then nothing is required except that $V_p$ is prefix-free.
This permits to construct the following machine $V$:
$V(r) = V_p(q)$ iff $r = pq$, $U(p)$ is defined
and $V_p(q)$ is defined; if $r$ cannot be split in $p,q$
this way then $V(r)$ is undefined; note that the splitting of $r$
into $p,q$ is unique whenever it is possible. The main properties
of $V$ are the following:
\begin{itemize}
\item $V$ is prefix-free;
\item for all $n$ and all $m>n$ there are $p,q$ such that
      $V(pq) = 2^m+n$, $H(n)=|p|$, $U(p) = n$ and $|pq| \leq H(2^m+n)+c_1$.
\end{itemize}
Based on $V$ one constructs a further prefix-free machine $W$ such
that $W(r) = z$ iff
there are $x,y,m,n,p,q,k$ such that the following conditions
are satisfied:
\begin{itemize}
\item $z = xy$;
\item $m>n+k+1$;
\item $|xy| = 2^m+n$ and $|y|=2k$;
\item $r = pq1^k0x$;
\item $U(p)$ is defined and takes the value $n$;
\item $V_p(q)$ is defined and takes the value $2^m+n$;
\item the binary value $bv(y)$ satisfies that
      $|p| \in P_{m,bv(y)}$.
\end{itemize}
Here $bv(y)$ is the binary value of $y$, for example, $bv(000101) = 5$.
Note that $x,y,m,n,p,q,k$ depend uniquely on $r$ whenever $W(r)$
can be defined by an appropriate choice of the parameters.

One can see from the definition that $W$ is prefix-free.
Hence there is a constant $c_2$
such that $H(W(r)) \leq |r|+c_2$ for all $r$ in the domain of $W$.

Furthermore, note that $H(F(2^m+n)) \geq H(2^m+n)+2^m+n-c_3$ for some
constant $c_3$. To see this, note that the sum of $2^{-H(u)}$ over all
$u \in \{0,1\}^{\ell}$ is at most $2^{c_4-H(\ell)}$ for some
constant $c_4$ independent of $\ell$. As there are $2^{\ell}$ such
$u$, it holds that $H(u) \geq H(\ell)+\ell-c_4$ for at least one
of these $u$. Now one can take $c_3 = c_4$ and use that $F$ takes
a string of maximal prefix-free Kolmogorov complexity to get the
desired statement.

Fix the value of the parameter $k$ from now onward as
\[
   k = c_1 + c_2 + c_3 + 2
\]
and note that for all $r$ in the domain
of $W$ having this fixed parameter $k$ it holds that
$H(W(r)) < H(F(|W(r)|))$.

Now the second part is completed by showing that whenever
$F(2^m+n) = xy$ with $m > n+k+1 \wedge |y| = 2k$ then
$H(n) \notin P_{m,bv(y)}$. So assume by way of contradiction
that $m > n+k+1$, $F(2^m+n) = xy$, $|y| = 2k$ and $H(n) \in P_{m,bv(y)}$.

Let $p$ be a program with $U(p)=n \wedge |p|=H(n)$.
Let $q$ be such that $2^m+n = V(pq)$ and $|q| \leq H(2^m+n)+c_1-|p|$;
the existence of such $q$ had been shown above when constructing
the machine $V_p$. One can verify that for the input $r = pq1^k0x$
and $z = xy$ the computation $W(r)$ converges to $z$ as in the definition
of $W(r)$, the first three search-conditions on $r,z$ are satisfied
by the choice of the above parameters, $p,q$ are selected such that
the fourth and fifth search-condition are satisfied and the assumption
$H(n) \in P_{m,bv(y)}$ gives that the sixth search-condition is
satisfied.

Now a contradiction is derived by showing that the two conditions
on $H(xy)$ are not compatible. On one hand, $r = pq1^k0x$ and
\[
|r| = |pq|+|1^k|+|0|+|x| \leq (H(2^m+n)+c_1)+k+1+|x| 
\]
\[
=
(H(2^m+n)+c_1)+k+1+(2^m+n-2k) 
\]
\[= H(2^m+n)+2^m+n+1+c_1-k
\]
\[
\leq H(2^m+n)+2^m+n-c_2-c_3-1;
\]
hence $H(xy) = H(W(r)) \leq |r|+c_2 \leq H(2^m+n)+2^m+n-c_3-1$.
On the other hand, $H(xy) = H(F(2^m+n)) \geq H(2^m+n)+2^m+n-c_3$.
This contradiction establishes that it does not happen
that $H(n) \in P_{m,bv(y)}$ for any $n$ and $m > n+k+1$ with
$y$ being the last $2k$ bits of $F(2^m+n)$.

Third, one can run the following $A$-recursive algorithm to determine
for any given $n$ a set of up to $4^k-1$ elements which
contains $H(n)$ by the following algorithm. Here $c_5$ is a constant
such that $H(n) \leq n+c_5$ for all $n$.
\begin{itemize}
\item Let $E = \{0,1,\ldots,n+c_5\}$ and $m = n+k+2$.
\item While $|E| \geq 4^k$ Do Begin $m=m+1$, \\
      Determine the string $y$ consisting of the last $2k$ bits of
      $F(2^m+n)$  \\
      and update $E = E - P_{m,bv(y)}$ End.
\item Output $E$.
\end{itemize}
This algorithm terminates since whenever $|E| \geq 4^k$ at some stage
$m$ then there is $o > m$ such that the first $4^k$ members of $P_o$
all intersect $E$ and one of them will be removed so
that $E$ loses an element in one of the stages $m+1,\ldots,o$.
Thus the above algorithm computes relative to $A$ for input $n$
a set of up to $4^k-1$ elements containing $H(n)$.
By a result of Beigel, Buhrman, Fejer, Fortnow, Grabowski,
Longpr\'e, Muchnik, Stephan and Torenvliet \cite{Muchnik:04},
such an $A$-recursive algorithm can only exist if $K \leq_T A$.
\end{proof}

\begin{remark} \rm
Calude \cite{Calude:05} had circulated the following question:
If $A$ is an infinite set of strings of maximal $H$-complexity, that is,
if $A$ satisfies
\[
   \forall x \in A\,\forall y\,[|y| = |x| \Rightarrow H(y) \leq H(x)],
\]
is then $K \leq_T A$? The question remains open until today, but
the above theorem gives a partial answer to this question as it
shows that $K \leq_T A$ is true for all sets $A$ of maximal
$H$-complexity which contain at least one string of each length.

Nies~\cite{Nies:04} pointed out to the authors that one might study
the analogue of incompressible strings in the sense that one looks at
functions $F$ producing strings of length $n$ and approximate
complexity $n+H(n)$. More precisely, the proof of the above
theorem also shows the more general result that
for any oracle $A$, $A \geq_T K$ if and only if
there is a function $F \leq_T A$ and a constant $c$ such that
for all $n$, $F(n) \in \{0,1\}^n$ and $H(F(n)) \geq n+H(n)-c$.

A related question to the one of Calude is whether there is an infinite
set $B$ such that $K \not\leq_T \{\langle n,H(n)\rangle: n \in B\}$.
If such a set exists, one can consider a constant $c$ such that
there is a set $A$ which contains for each $b \in B$ exactly
one $a \in \{0,1\}^b$ with $H(a) \geq H(b)+b-c$. Given $B$ and $c$,
such a set $A$ can be constructed relative to any oracle which is
PA-complete relative to $B$; note that such oracles need not be
above $K$. Hence Nies' version of the question of Calude has a
negative answer in the case that this set $B$ exists.
\end{remark}

\section{Characterizing high or DNR degrees}

\noindent
In this section various characterizations are obtained for
when a Turing degree is high or DNR, in terms of what are called eventually different functions.
The study of such functions originates in
set theory where set theorists had defined that a function
$f: \omega \to \omega$ is \emph{eventually different} iff for each
$g:\omega\to\omega$ in the ground model, $\{x:f(x)=g(x)\}$ is
finite. A computability-theoretic analogue is obtained by
replacing the ground model by the set of recursive functions. The
corresponding Turing degrees form the union of the high and the
DNR degrees and also admits a characterization in terms of
upper bounds on Kolmogorov complexity. Note that this analogue
is a relaxed version of SDNR functions as every SDNR function
is eventually different from every partial-recursive function
and every function eventually different from all partial-recursive
ones is DNR.

Theorem \ref{SNRisDNRorhigh} has been applied by
Stephan and Yu \cite{StephanYu},
and Greenberg and Miller \cite{GreenbergMiller}.

\begin{theorem}\label{SNRisDNRorhigh}
The following statements are equivalent:

\begin{enumerate}
\item[\rm(1)] $A$ computes a function $f$ that is eventually always
different from each recursive function.

\item[\rm(2)] $A$ computes a function $g$ such that either $g$ dominates
each recursive function or $g$ is eventually always different from
each partial recursive function.

\item[\rm(3)] $A$ is of high or DNR Turing degree.

\item[\rm(4)] $A$ computes an unbounded function $f$ which is dominated
by all recursive upper bounds on $C$;

\item[\rm(5)] $A$ computes a function $F$ such that for every
total recursive $V$ with domain $\{0,1\}^*$ and almost all $n$,
$F(n) \notin \{V(p) : p \in \{0,1\}^* \wedge |p| < n\}$.

\end{enumerate}
\end{theorem}

\begin{proof}
(1) implies (2): If $A$ has high degree then $A$ computes a function
$g$ dominating all recursive functions (we can either take this as our
definition of high degree, or invoke Martin's Theorem from 1966) and
(2) is satisfied. So
assume that $A$ does not have high degree and let $f \leq_T A$ be
eventually different from all total recursive functions.
For a contradiction suppose $\varphi_d$ is some partial recursive
function and $f(x)=\varphi_d(x)$ on infinitely many inputs $x$
in the domain of $\varphi_d$. Let $p$ be a function such that for
each $n$, \emph{there are $n+1$ many $x$ for which
$f(x)=\varphi_{d,p(n)}(x)$}. (Here $\varphi_{e,s}(x)$ is the value of
$\varphi_e(x)$ after $s$ steps.) 

Clearly $p\le_T f$. Hence as $f$ is
not of high degree, there is a recursive nondecreasing function
$q$ such that for infinitely many $n$, $q(n)\ge p(n)$. Now define
a total recursive function $\varphi_e$ by
$\varphi_e(n)=\varphi_{d,q(n)}(n)$ if this computation halts, and
$\varphi_e(n)=0$ otherwise. Now suppose $q(n)\ge p(n)$. Then for
some $k\ge n$, $\varphi_{d,q(n)}(k)=f(k)$. So
$\varphi_{d,q(k)}(k)=f(k)$ and hence $\varphi_e(k)=f(k)$. Since
there are infinitely many such $n$, there are infinitely many such
$k$, and hence $f$ agrees with a total recursive function on
infinitely many inputs.

(2) implies (3): Let $g \leq_T A$ have the desired properties.
If $f$ dominates every recursive function then $A$ has high
Turing degree. If $g$ is eventually different from all partial
recursive functions then consider any $e$ such that
$\varphi_e(x)=\varphi_x(x)$ for each $x$.
Then a finite modification of $f$ is DNR, hence $A$ has DNR degree.

(3) implies (4): First assume the case that $A$ is of
high Turing degree. Then there is a function $g \leq_T
A$ which grows faster than every recursive function. Taking $U$ to
be the universal machine, then $f(x)$ is the length of the
shortest program $p$ such that $U(p) = x$ within $g(x)$ steps. If
$\tilde C$ is a recursive upper bound of $C$ then let $t(x)$ be
the time to compute $U(p)$ for the fastest program $p$ of length
up to $\tilde C(x)$ with $U(p) = x$. The function $t$ is recursive
and dominated by $g$, thus $f(x) \leq \tilde C(x)$ for almost
all~$x$. So $\tilde C$ dominates $f$.

Second assume the case that $A$ is of DNR Turing degree.
Then by Theorem~\ref{th:complexchar}~(4) there is a function
$g \leq_T A$ such that $C(g(n)) \geq n$ for all $n$.
Without loss of generality $f(n)$ is a string of length
at least $n$ for every $n$. Now define for every $x \in \{0,1\}$
the value $f(x)$ as the maximum of all $m \leq n$ such that
either $m=0$ or $x = g(m)$. Clearly $f$ is unbounded as
$f(g(m)) \geq m$ for all $m$. Furthermore, $f$ is dominated
by $C$ and hence also by all upper bounds $\tilde C$ of $C$.

(4) implies (5): Let $f$ be the given lower bound. Now define
$F(n)$ to be the length-lexicographically first string $x$ with
$f(x)\ge 2n$ which exists by the assumption that $f$ is unbounded.

Let $V$ be recursive with domain $\{0,1\}^*$. There is a constant
$c$ such that $C(V(p)) \leq |p|+c$ for all programs $p$ and $C(x)
\leq |x|+c$ for all $x$. Now let $\tilde C$ be the minimum of
$|p|+c$ for all programs $p$ with either $|p| = |x|$ or $V(p) =
x$. $\tilde C$ is a recursive upper bound for $C$. Thus, for
sufficiently large $n$ and all $p$ of length up to $n$,
$\tilde C(F(n)) \geq 2n>n+c$ and so $V(p) \neq F(n)$.

(5) implies (1): Given any function $h$, define for all strings
$p$ of length $n-1$ that $V(p) = h(n)$. It follows that the
function $F$ differs from $h$ almost always and so (1) is
satisfied.
\end{proof}

\noindent
A similar result can be obtained for the Turing degrees of
Peano complete or high sets.

\begin{theorem}
The following statements are equivalent for any set $A$:
\begin{enumerate}
\item[\rm(1)] $A$ has either high or PA-complete Turing degree;

\item[\rm(2)] $A$ computes a function $B$ such that
$\forall n\,(\,|\{x \in \{0,1\}^*: B(x) = n\}| \leq 2^n)$ and $B$ is
dominated by all recursive upper bounds of $C$;

\item[\rm(3)] $A$ computes a function $F$ mapping every $n$ to a
string of length $n$ such that for every recursive upper bound
$\tilde C$ on $C$ and for almost all $n$, $\tilde C(F(n)) \geq n$;

\item[\rm(4)] $A$ computes a $\{0,1\}$-valued function $g$ such that
for all infinite recursive subsets $R$ of the domain of the diagonal
function, $\forall^{\infty} n \in R\,(g(n) \neq \varphi_n(n))$.

\end{enumerate}
\end{theorem}

\begin{proof}
(1) implies (2): If $A$ has PA-complete degree, then $B$ exists
by Theorem~\ref{pacharacterization}. If $A$ has high degree, then
let $d \leq_T A$ be a function which dominates all
recursive functions and let $B(x) = C_{d(|x|)}(x)$. $B$ is an upper bound
of $C$ and thus satisfies the cardinality condition. Furthermore, if
$\tilde C$ be a recursive upper bound on $C$, then the function
mapping $n$ to the first $s$ such that $C_s(x) \leq \tilde C(x)$
for all $x \in \{0,1\}^*$ with $|x| \leq n$ is recursive and thus
dominated by $d$. It follows that $B(x) \leq \tilde C(x)$ for
almost all $x$ and so (2) is satisfied.

(2) implies (3): Take $B$ as specified for (2) and let $F(n)$ be
the lexicographic first string $x$ of length $n$ with $B(x) \geq n$.
This string exists since there are at most $2^n-1$ many strings
$y$ with $B(y) < n$. Note that $F \leq_T A$ since $B \leq_T A$.
Since every recursive upper bound $\tilde C$ dominates $B$,
condition (3) is satisfied.

(3) implies (1): Assume that $A$ does not have PA-complete Turing
degree. Then by Theorem~\ref{pacharacterization}, there are infinitely
many $n$ with $C(F(n)) < n$. Then there is an increasing $A$-recursive
function $d$ such that for all $n$ there is $m \geq n$ with $C_{d(n)}(F(m)) < m$.
In particular, if $h$ is also increasing and $h(n) \geq d(n)$ infinitely
often, then $C_{h(m)}(F(m)) < m$ for infinitely many $m$. So the mapping
from $x$ to $C_{h(|x|)}(x)$ cannot be recursive and $h$ cannot be a
recursive function. Thus $d$ dominates every recursive function
and $A$ has high Turing degree.

(1) implies (4): If $A$ has PA-complete degree then it is well-known
that there is a $\{0,1\}$-valued DNR function $g \leq_T A$.
If $A$ has high degree then one can again take an
$A$-recursive function $d$ dominating all recursive ones and consider
the function $g$ with
$g(n) = 1 \Leftrightarrow \varphi_{n,d(n)}(n) \da = 0$.
If $R$ is a recursive subset of the domain of the diagonal function
then $d$ dominates the time which $\varphi_n(n)$ needs to converge
on inputs from $R$ and thus $g(n) = \varphi_n(n)$ only for finitely
many $n$ in $R$.

(4) implies (1): This is similar to the implication from (3) to (1).
Assume that $A$ does not have PA degree. Let $g \leq_T A$ be
$\{0,1\}$-valued. There are infinitely many $n$ in the domain of the diagonal
function with $g(n) = \varphi_n(n)$. There is an increasing $A$-recursive
function $d$ such that for all $n$ there is $m \geq n$ with
$g(m)=\varphi_{m,d(n)}(m)$. For given increasing, recursive function $h$
let $R = \{n: \varphi_{n,h(n)}(n) \da\}$. Whenever $h(n) \geq d(n)$
then there is $m \geq n$ with $\varphi_{m,d(n)}(m) \da = g(m)$.
Furthermore, $m \in R$ since $d(n) \leq h(n) \leq h(m)$. Since
$g$ is correct on only finitely many elements of $R$, $h(n) < d(n)$ for
almost all $n$ and $d$ dominates every recursive function. Thus $A$
has high Turing degree.
\end{proof}

\begin{theorem}
If a set $A$ has high Turing degree then there is a function $F \leq_T A$
mapping every $n$ to a string of length $n$ such that
\[
\forall x \in \{0,1\}^n\,(H(x) \leq \tilde H(F(n)))
\]
for every recursive upper bound $\tilde H$ of $H$ and almost all $n$.
\end{theorem}

\begin{proof}
Let $A$ have high degree. There is a function
$f \leq_T A$ which dominates all recursive functions.
Let $F(n)$ be the lexicographically first $x \in \{0,1\}^n$ for
which $H_{f(n)}(x)$ is maximal. Now, let $\tilde H$ is a recursive upper
bound on $H$. The function mapping $n$ to the first $s$ such that
$H_s(x) \leq \tilde H(x)$
for all $x \in \{0,1\}^n$ is recursive and thus
dominated by $f$. It follows that $H_{f(n)}(F(n)) \leq \tilde H(F(n))$
for almost all $n$. On the other hand, $H(x) \leq H_{f(n)}(F(n))$
for all $n$, so the statement of the theorem is satisfied in the case
that $A$ has high Turing degree.
\end{proof}

\section{R.e.\ traceable sets}

\noindent
It was shown in \cite{Kjos-Hanssen.Nies.ea:05} that a set $A$ is
r.e.\ traceable if and only if every Martin-L\"of random set is Schnorr
random relative to $A$. We remind the reader of the definitions. (Recall that $W_n$ is the $n$th r.e.\ set, and $D_n$ the finite set with canonical index $n$.)

\begin{definition}
A set $A$ is \emph{r.e.\ traceable} if there is a recursive function $p$ such that for all $f\le_T A$, there is a recursive function $g$ such that for all $n$, $f(n)\in W_{g(n)}$ and $W_{g(n)}$ has at most $p(n)$ elements. Similarly, $A$ is \emph{recursively traceable} if the same statement holds with the r.e.\ set $W_{g(n)}$ replaced by the canonically finite set $D_{g(n)}$.
\end{definition}

We now characterize r.e.\ traceable sets as being ``uniformly very far from DNR''. A similar characterization holds for recursively traceable sets, and shows that recursively traceable sets compute no eventually different function.

\begin{theorem}\label{oct1-2004}
The following statements are equivalent for any set $A$:
\begin{enumerate}
\item[\rm(1)] $A$ is r.e.\ traceable.

\item[\rm(2)] There is a fixed recursive function $z(n)$ such that for
each $f\le_T A$, and almost every $n$, the set
$\{x:f(x)=\varphi_x(x)\}$ has at least $n$ elements below $z(n)$.
\end{enumerate}

The following statements are equivalent for any set $A$:
\begin{enumerate}
\item[\rm(3)] $A$ is recursively traceable.

\item[\rm(4)] There is a fixed recursive function $g$ such that for
each $f\le_T A$, there is a recursive function $\varphi_e$ such
that for almost every $n$, the set $\{x:f(x)=\varphi_e(x)\}$
has at least $n$ elements below $g(n)$.
\end{enumerate}
\end{theorem}

\begin{proof}
(1) implies (2): Suppose $A$ is r.e. traceable via the recursive
function $p(n)$, and let $f\le_T A$. Let
\[
q(n)=\max\{r(i,e,n)\mid i<p(n),e<n\}
\] 
where $r(i,e,n)$ is primitive
recursive such that 
\[
	\varphi_{r(i,e,n)}(r(i,e,n))\simeq\\ \text{ the
$i$th member of }W_{\varphi_e(n)}\text{ in order of
enumeration.} 
\]  
Since $f\le_TA$ and $q$ is recursive, $f\circ q\le_TA$,
so let $g(n)$ be recursive such that for all $n$, $f(q(n))\in W_{g(n)}$
and $|W_{g(n)}|\le p(n)$.  Let $e$ be an index of $g$, i.e.,
$\varphi_e(n)=g(n)$ for all $n$.

Suppose for all $n>e$ we have
$f(q(n))\ne\varphi_{r(i,e,n)}(r(i,e,n))\simeq$ the $i$th member
of $W_{\varphi_e(n)}=W_{g(n)}$, for all $i<p(n)$.  Since
$|W_{g(n)}|\le p(n)$, it follows that $f(q(n))\notin W_{g(n)}$, a
contradiction. So $f(q(n))$ must have been equal to
$\varphi_{r(i,e,n)}(r(i,e,n))$ for some $i<p(n)$ and all $n>e$. But
this gives a bounding function $z(n)$ witnessing that $f$ is not SDNR.
This is now easily translated into a bounding function witnessing that
any function recursive in $A$ is not even DNR, via the proof of Theorem \ref{nov192006}. 

(2) implies (1): Given $f\le_T A$, let $\hat f(x)=(f(0),\ldots,f(x))$.
By assumption $\hat f(x)=\varphi_x(x)$ for $y+1$ many $x$ below
$z(y+1)$. Hence there is such an $x$ with $x\ge y$, and so 
\[
	f(y)\in T_y:=\{(\varphi_x(x))_y: x\ge y, x<z(y+1)\},
\]
for almost every
$y$. So by
modifying the trace finitely, it holds for every $y$. The size of
$T_y$ is bounded by ${z(y+1)}$, since $T_y$ contains at most one number for each $x<z(y+1)$. 

(3) implies (4): The above argument gives the conclusion
that no recursively traceable set computes a function that agrees
with each recursive function only finitely often. Indeed, if
$W_{g(n)}$ has recursive size then we can bound the running time of
$\varphi_{r(i,e,n)}$ to produce a partial recursive function with
recursive domain, which hence has a total recursive extension.

(4) implies (3): Follows the proof that (2) implies (1). If
$(f(0),\ldots,f(x))=\varphi_e(x)$ then
\[
	f(y)\in T_y:=\{(\varphi_e(x))_y: x\ge y, x<g(y+1)\},
\]
for almost every
$y$, and the size of $T_y$ is bounded by ${g(y+1)}$.
\end{proof}

\section*{Acknowledgments}

\noindent
The authors would like to thank Cristian S.\ Calude and Andr\'e Nies, 
for correspondence and helpful comments to improve the proof
of Theorem~\ref{th:calude}; and Stephen G.~Simpson, for the idea to use SDNR functions to express the
proof of Theorem~\ref{oct1-2004}; and A.\ Khodyrev, who pointed out that one can give an easier proof of the equivalence of complex and autocomplex with DNR by going via SDNR functions. 

\bibliographystyle{amsplain}

\begin{bibdiv}
\begin{biblist}

\bib{AKLS}{article}{
   author={Ambos-Spies, Klaus},
   author={Kjos-Hanssen, Bj{\o}rn},
   author={Lempp, Steffen},
   author={Slaman, Theodore A.},
   title={Comparing DNR and WWKL},
   journal={J. Symbolic Logic},
   volume={69},
   date={2004},
   number={4},
   pages={1089--1104},
   issn={0022-4812},
   review={\MR{2135656 (2006c:03061)}},
}

\bib{Muchnik:04}{article}{
   author={Beigel, Richard},
   author={Buhrman, Harry},
   author={Fejer, Peter},
   author={Fortnow, Lance},
   author={Grabowski, Piotr},
   author={Longpre, Luc},
   author={Muchnik, Andrej},
   author={Stephan, Frank},
   author={Torenvliet, Leen},
   title={Enumerations of the Kolmogorov function},
   journal={J. Symbolic Logic},
   volume={71},
   date={2006},
   number={2},
   pages={501--528},
   issn={0022-4812},
   review={\MR{2225891 (2007b:68086)}},
}

\bib{Calude:05}{unpublished}{,
	author={Cristian S.\ Calude}, 
	note={Private communication}, 
	year={2005},
}

\bib{GreenbergMiller}{article}{
   author={Greenberg, Noam},
   author={Miller, Joseph S.},
   title={Lowness for Kurtz randomness},
   journal={J. Symbolic Logic},
   volume={74},
   date={2009},
   number={2},
   pages={665--678},
   issn={0022-4812},
   review={\MR{2518817 (2010b:03050)}},
}

\bib{Jockusch:89}{article}{
   author={Jockusch, Carl G., Jr.},
   title={Degrees of functions with no fixed points},
   conference={
      title={Logic, methodology and philosophy of science, VIII},
      address={Moscow},
      date={1987},
   },
   book={
      series={Stud. Logic Found. Math.},
      volume={126},
      publisher={North-Holland},
      place={Amsterdam},
   },
   date={1989},
   pages={191--201},
   review={\MR{1034562 (91c:03036)}},
}

\bib{Kanovich:70}{article}{
   author={Kanovi{\v{c}}, M. I.},
   title={The complexity of the enumeration and solvability of predicates},
   language={Russian},
   journal={Dokl. Akad. Nauk SSSR},
   volume={190},
   date={1970},
   pages={23--26},
   issn={0002-3264},
   review={\MR{0262084 (41 \#6694)}},
}

\bib{Kanovich:69}{article}{
   author={Kanovi{\v{c}}, M. I.},
   title={The complexity of the reduction of algorithms},
   language={Russian},
   journal={Dokl. Akad. Nauk SSSR},
   volume={186},
   date={1969},
   pages={1008--1009},
   issn={0002-3264},
   review={\MR{0262082 (41 \#6692)}},
}

\bib{Kjos-Hanssen.Nies.ea:05}{article}{
   author={Kjos-Hanssen, Bj{\o}rn},
   author={Nies, Andr{\'e}},
   author={Stephan, Frank},
   title={Lowness for the class of Schnorr random reals},
   journal={SIAM J. Comput.},
   volume={35},
   date={2005},
   number={3},
   pages={647--657 (electronic)},
   issn={0097-5397},
   review={\MR{2201451 (2006j:68051)}},
   doi={10.1137/S0097539704446323},
}


\bib{LiVitanyi2}{book}{
   author={Li, Ming},
   author={Vit{\'a}nyi, Paul},
   title={An introduction to Kolmogorov complexity and its applications},
   series={Graduate Texts in Computer Science},
   edition={2},
   publisher={Springer-Verlag},
   place={New York},
   date={1997},
   pages={xx+637},
   isbn={0-387-94868-6},
   review={\MR{1438307 (97k:68086)}},
}


\bib{LiVitanyi3}{book}{
   author={Li, Ming},
   author={Vit{\'a}nyi, Paul},
   title={An introduction to Kolmogorov complexity and its applications},
   series={Texts in Computer Science},
   edition={3},
   publisher={Springer},
   place={New York},
   date={2008},
   pages={xxiv+790},
   isbn={978-0-387-33998-6},
   review={\MR{2494387 (2010c:68058)}},
}



\bib{Miller:02}{thesis}{
	author={Joseph~Stephen Miller},
	title={Pi-0-1 classes in computable analysis and topology},
	note={PhD thesis, Cornell University}, 
	year={2002},
}


\bib{Nies:04}{unpublished}{
	author={Andr\'e Nies},
	note={Private communication},
	year={2004},
}

\bib{StephanYu}{article}{
   author={Stephan, Frank},
   author={Yu, Liang},
   title={Lowness for weakly 1-generic and Kurtz-random},
   conference={
      title={Theory and applications of models of computation},
   },
   book={
      series={Lecture Notes in Comput. Sci.},
      volume={3959},
      publisher={Springer},
      place={Berlin},
   },
   date={2006},
   pages={756--764},
   review={\MR{2279138 (2007h:03073)}},
}


\end{biblist}
\end{bibdiv}

\end{document}